\documentclass[11pt,reqno]{amsart}
\usepackage{amsmath,amsthm,amssymb,enumerate}
\usepackage{hyperref}
\usepackage{mathrsfs}
\usepackage{float}
\usepackage{graphicx}

\usepackage{color}

\newtheorem{theorem}{Theorem}[section]
\newtheorem{lemma}{Lemma}[section]

\newtheorem{corollary}[theorem]{Corollary}

\theoremstyle{definition}
\newtheorem{assumption}{Assumption}[section]
\newtheorem{definition}{Definition}[section]

\theoremstyle{remark}
\newtheorem{remark}{Remark}[section]

\newcommand\bR{\mathbb{R}}

\newcommand*{\beq}{\begin{equation}}
\newcommand*{\eeq}{\end{equation}}

\newcommand{\R}{\mathbb{R}}
\newcommand{\bit}{\begin{itemize}}
\newcommand{\eit}{\end{itemize}}
\newcommand{\osc}{\text{osc}}
\newcommand\D{\partial}

\begin{document}



\title[K. Dareiotis and   M. Gerencs\'er]{Local $L_\infty$-estimates, weak Harnack inequality, and stochastic continuity of solutions of SPDEs}


\author[K. Dareiotis]{Konstantinos Dareiotis}
\address[K. Dareiotis]{Uppsala University, Sweden}
\email{konstantinos.dareiotis@math.uu.se}

\author[M. Gerencs\'er]{M\'at\'e Gerencs\'er}
\address[M. Gerencs\'er]{University of Edinburgh, United Kingdom}
\email{m.gerencser@sms.ed.ac.uk}

\begin{abstract}
We consider stochastic partial differential equations under minimal assumptions: the coefficients are merely bounded and measurable and satisfy the stochastic parabolicity condition. In particular, the diffusion term is allowed to be scaling-critical. We derive local supremum estimates with a stochastic adaptation of De Giorgi's iteration and establish a weak Harnack inequality for the solutions. The latter is then used to obtain pointwise almost sure continuity.

\end{abstract}

\maketitle


%￼\newtheorem{theorem}{Theorem}[section]
%￼\newtheorem{corollary}{Corollary}[section]
%%￼\newtheorem{assumption}{Assumption} [section]
%￼\newtheorem{definition}{Definition} [section]
%\newtheorem{lemma}{Lemma} [section]
%\newdefinition{remark}{Remark} [section]
%\newproof{proof}{Proof}
%\newproof{pofb}{Proof of Theorem \ref{thm: supremum estimate}}
%\newproof{pofc}{Proof of Theorem \ref{thm: continuity}}
%\newproof{pofh}{Proof of Theorem \ref{thm: harnack}}

\section{Introduction}

Harnack inequalities, introduced by \cite{Har}, provide a comparison of values at different points of nonnegative functions which satisfy a partial differential equation (PDE). Inequalities of this type have a vast number of applications, in particular, they played a significant role in the study of PDEs with discontinuous coefficients in divergence form. This is the celebrated De Giorgi-Nash-Moser theory (\cite{DG}, \cite{NASH}, \cite{MOS}), in which H\"older continuity of the solutions is established. Later, by using a weaker version of Harnack's inequality, a simpler proof in the parabolic case was given in \cite{KRU}. Harnack inequality and H\"older  estimate for equations in non-divergence form, also known as the Krylov-Safonov estimate, was proved in \cite{KRSAF} and \cite{SAF}. Since then, similar results have been proved for more general equations, including for example integro-differential operators of L\'evy type (see \cite{CAR}) and singular equations (see \cite{DIBE} and references therein). 

It is well known (see e.g. \cite{KSEE}, \cite{KrylovLP}) that the stochastic partial differential equations (SPDEs)
\begin{equation}\label{SPDE}
du_t=L_tu_t \,dt+M^k_t u_t \, dw^k_t,
\end{equation}
where $M^k$ are first order differential operators, are in many ways the natural stochastic extensions of parabolic equations $du_t=L_tu_t \, dt$. It is therefore also natural to ask whether the above mentioned results, fundamental in deterministic PDE theory, have stochastic counterparts. That is, what properties can one obtain for weak solutions of \eqref{SPDE} without posing any smoothness assumptions on the coefficients? Note also that with bounded coefficients the diffusion term in \eqref{SPDE} is critical to the parabolic scaling, and hence the question above fits in the recent activity in parabolic regularity
with critical lower order terms, see e.g. \cite{Chen} and its references. In some recent works regularity results have been obtained, but only for equations with at most zero order $M$, that is, with subcritical noise, for variants of this problem we refer to \cite{Kim}, \cite{HWW}, \cite{DDH}, and \cite{KUK}. The methods in all of these works rely strongly on the absence of the derivatives in the noise, in which case the difficulty coming from the lack of regularity of the coefficients can be separated from the stochastic nature of the equation and can be essentialy reduced to the deterministic case. In particular, adaptation of the classical techniques of \cite{DG}, \cite{NASH}, \cite{MOS} to the stochastic setting is not required, which is indeed what the scaling heuristic would suggest. 

Concerning equations of the general form and under minimal assumptions - boundedness, measurability, and ellipticity - on the coefficients, few results are known. They were considered in \cite{MAKO} (see also \cite{Qiu2}) and, in a backward setting, in \cite{BSPDE}, where global boundedness of the solutions was proved. In the present paper,  we prove local $L_\infty$-estimates for certain functions of the solutions, in terms of the corresponding $L_2$-norms, by using a stochastic version of De Giorgi's iteration. By virtue of  these estimates,  following the approach of \cite{KRU}, we establish a stochastic version of the aforementioned weak Harnack inequality in Theorem \ref{thm: harnack}. Here ``weak'' stands for that in order to estimate the minimum of a nonnegative solution $u$, not only the maximum of $u$ is required to be bounded from below by $1$, but $u$ itself on a positive portion of the domain. For deterministic equations by elementary arguments one can deduce H\"older continuity from such a weak Harnack inequality. These considerations however are quite sensitive to the measurability problems arising with the presence of stochastic terms, and therefore we need a far less straightforward argument to prove stochastic continuity of the solutions, which is formulated in Theorem \ref{thm: continuity}. We note that Harnack inequalities for solutions of SPDEs - not to be confused with Harnack inequalities for the transition semigroup of SPDEs, for which we refer the reader to \cite{Wang} and the references therein - have not been previously established even for equations with smooth coefficients. 

Let us introduce the notations used throughout the paper. Let $d\geq1$, and for $R \geq 0$ let $B_R=\{x\in\R^d:|x| <  R\}$, $G_R=[4-R^2,4]\times B_R$, and $G:=G_2$. $\mathcal{B}(B_R)$ will denote the Borel $\sigma$-algebra on $B_R$. Subsets of $\R^{d+1}$ of the form $J\times (B_R+x)$, where $J$ is a closed interval in $[0,4]$ and $x \in \bR^d$, will be referred to as cylinders. If $A$ is a set, $I_A$ will denote the indicator function of A. The inner product in $L_2(B_2)$ will be denoted by $(\cdot,\cdot)$. The set of all compactly supported smooth functions on $B_{R}$ will be denoted by $C_c^{\infty}(B_{R})$. The space of $L_2(B_{R})$-functions whose generalized derivatives of first order lie in $L_2(B_{R})$ is denoted by $H^1(B_{R})$, while the completion of $C_c^{\infty}(B_{R})$ with respect to the $H^1(B_{R})$ norm is denoted by $H^1_0(B_{R})$. For $p\in [1,\infty]$ and a subset $A$ of $\R^d$ or $\R^{d+1}$, the norm in $L_p(A)$ will be denoted by $|\cdot|_{p,A}$ and $\|\cdot \|_{p,A}$, respectively. By $\inf,\sup, \osc,$  we always mean essential ones. We fix a complete probability space $(\Omega, \mathscr{F}, P)$ and take a 
right-continuous filtration $(\mathscr{F}_t)_{t\geq0}$, 
such that $\mathscr{F}_0$ contains all 
$P$-zero sets, and a sequence of 
independent 
real valued  $\mathscr{F}_t$-Wiener processes  $\{w^k_t\}_{k=1}^{\infty}$. The predictable $\sigma$-algebra on $\Omega \times [0,4]$ is denoted by $\mathscr{P}$. Constants in the calculations are usually denoted by $C$, and, as usual, may change from line to line. The summation convention with respect to repeated integer-valued indices will be in effect.

The rest of the paper is organized as follows. In Section 2 we formulate the assumptions and state the main results. In Section 3 we present some preliminary results, which are then used in the proofs of the main results in Section 4.

\section{Formulation and main results}

The operators in \eqref{SPDE} are assumed to be of the form
$$
L_t\varphi=\D_i(a^{ij}_t\D_j\varphi),\,\,\,M_t^k\varphi=\sigma_t^{ik}\D_i\varphi,
$$
where we pose the following assumption on the coefficients throughout the paper.
\begin{assumption}\label{as0}
For $i , j \in \{1,...,d\}$, the functions $a^{ij}=a^{ij}_t(x)(\omega)$ and $\sigma^i =(\sigma^{ik}_t(x)(\omega))_{k=1}^\infty$ are $\mathscr{P}\times\mathcal{B}(B_2)$-measurable functions
on $\Omega \times [0, \infty) \times B_2$ with values in $\bR$ and $l_2$, respectively, bounded by a constant $K$, such that 
$$
(2a^{ij}-\sigma^{ik}\sigma^{jk})z_iz_j\geq\lambda|z|^2
$$
for a $\lambda>0$ and for any $z=(z_1,\ldots,z_d)\in\R^d.$
\end{assumption}
We will denote by $\mathcal{H}$ the set of all strongly continuous $L_2(B_2)$-valued predictable processes $u=(u_t)_{t\in[0,4]}$ on $\Omega \times [0,4]$ such that $u \in L_2([0,4], H^1(B_2))$ with probability 1.

\begin{definition}We say that $u$  is a solution of \eqref{SPDE}, 
if $u \in \mathcal{H}$ and  for each $\phi \in C_c^\infty(B_2)$, with probability one,
$$
(u_t,\phi)=(u_0,\phi)-\int_0^t(a^{ij}_t\D_i u_t, \D_j \phi )dt+\int_0^t (\sigma^{ik}_t \D_i u_t, \phi)dw^k_t,
$$
for all $t \in [0,4]$. 
\end{definition}

The class $\mathcal{H}$ is the ``right one'' to seek solutions in, in the sense that the classical theory (see e.g. \cite{KSEE}) guarantees the existence of a unique solution of \eqref{SPDE} in $\mathcal{H}$, when coupled with appropriate initial and boundary condition. Elements of $\mathcal{H}$ however, in general don't have any kind of spatial continuity (unless $d=1$), and are not even in $L_p$ for high values of $p$. 

Let us denote by $\mathcal{C}$ the set of twice continuously differentiable  functions $f$ from $\R$ to $\R$, such that  both $f'$ and $ff''$ are bounded.  The next is our first main result. 

\begin{theorem}\label{thm: supremum estimate}
Let $f\in\mathcal{C}$ such that $ff''\geq 0$, and let $u$ be a solution of \eqref{SPDE}. Then there exist positive constants $\delta,C, \hat{C}$, depending only on $d,\lambda,K$, such that for any $\alpha>0$ and $\kappa\geq 1$
\begin{itemize}
\item[(i)] $P(\|f(u)^+\|_{\infty,G_1}^2\geq\hat{C}\kappa\alpha,\|f(u)^+\|_{2,G_{3/2}}^2\leq\alpha)\leq C\kappa^{-\delta}$,
\item[(ii)]
$P(\|f(u)\|_{\infty,G_1}^2\geq\hat{C}\kappa\alpha,\|f(u)\|_{2,G_{3/2}}^2\leq\alpha)\leq C\kappa^{-\delta}.$
\end{itemize}
Let $f$ be as above, let $u$ be a solution of \eqref{SPDE} on $[s,r]\times B_2$, where  $0 \leq s< r \leq 4$, and suppose that $f(u)(s,\cdot)\equiv 0$. Then there exist positive constants $\delta,C, \hat{C}$, depending only on $d,\lambda,K$, such that for any $\alpha>0$ and $\kappa\geq1$
\begin{itemize}
\item[(iii)]$
P(\|f(u)^+\|^2_{\infty, [s,r]\times B_1}\geq  \hat{C} \kappa\alpha, \|f(u)^+\|^2_{2, [s,r]\times B_2}\leq\alpha)\leq C\kappa^{-\delta}
$,
\item[(iv)] $P(\|f(u)\|^2_{\infty, [s,r]\times B_1}\geq  \hat{C} \kappa\alpha, \|f(u)\|^2_{2, [s,r]\times B_2}\leq\alpha)\leq C\kappa^{-\delta}$.
\end{itemize}
\end{theorem}

To formulate the Harnack inequality, let $\eta\in(0,1)$, and denote by $\Lambda_\eta$ the set of functions $v$ on $[0,4]\times B_{2}$ such that $v\geq 0$ and 
$$
|\{x\in B_2|\,v_0(x)\geq 1\}|\geq\eta|B_2|.
$$
Let us recall the Harnack inequality (essentially) proved in \cite{KRU} : 
\emph{If $u$ is a solution of $du=\D_i(a^{ij}\D_ju)dt$ and $u\in\Lambda_{1/2}$, then $$\inf_{G_1}u\geq h$$ with $h=h(d,\lambda,K)>0$.} 
In the stochastic case clearly it can not be expected that such a lower estimate holds uniformly in $\omega$. It does hold, however, with $h$ above replaced with a strictly positive random variable, this is the assertion of our main theorem.

\begin{theorem}\label{thm: harnack}
Let $u$ be a solution of \eqref{SPDE} such that on an event $A \in \mathscr{F}$, $u\in \Lambda_{\eta}$. 
Then for any $N>0$  there exists a set $D \in \mathscr{F}$, with $P(D) \leq  CN^{-\delta} $, such that on $A \cap D^c$, 
$$
\inf_{(t,x)\in G_1} u_t(x) \geq e^{- N}.
$$
where  $C$ and $\delta$ depend only on $d$, $\lambda$, $\eta$, and $K$.
\end{theorem}
 Later on we will refer to the quantity $e^{-N}$ above as the lower bound corresponding to the probability $ CN^{-\delta}$. With the help of Theorem \ref{thm: harnack}, we obtain the following stochastic continuity result.

\begin{theorem}\label{thm: continuity}
Let $u$ be a solution of \eqref{SPDE} and $(t_0,x_0)\in(0,4)\times B_2$. Then $u$ is almost surely continuous at $(t_0,x_0).$
\end{theorem}

\begin{remark}
One advantage of the present setting with very mild assumptions is that the results trivially extend to quasilinear equations, that is, when  $Lu$ and $M^ku$ are replaced by  $\tilde{L}u=\D_i(a^{ij}(u)\D_ju)$, and $\tilde{M}^ku=\sigma^{ik}(u)\D_iu$  where the functions $a^{ij}(\cdot)$, $|\sigma^{i\cdot}(\cdot)|_{l_2}$ are bounded and $(\alpha^{ij}(\cdot))_{i,j=1}^d=(2\alpha^{ij}(\cdot)-\sigma^{ik}(\cdot)\sigma^{jk}(\cdot))_{i,j=1}^d$ takes values in the set $\{(\beta^{ij})_{i,j=1}^d:\forall z\in\R^d,\beta^{ij}z_iz_j\geq\lambda|z|^2\}$ for some $\lambda>0$. 
\end{remark}
\begin{remark}We only consider equations with higher order terms,  in the same spirit as in e.g. \cite{MOS1}. This reason for this is to focus on the stochastic aspect of the problem, the lower order terms with measurable and appropriately bounded (i.e. in a subcritical norm) coefficients can be easily treated, as exposed in detail and in great generality in \cite{LA}.
\end{remark}

\section{Preliminaries}

The first three lemmas might be considered standard in the context of stochastic processes and parabolic PDEs, respectively. For the sake of completeness we provide short proofs.
\begin{lemma}\label{lem: martingale 2}
Let $T>0$ and let $(m_t)_{t\in[0,T]}$ be a continuous local martingale starting from 0. Then for any  $\alpha>0$, and $\varkappa>0 $
\begin{equation*}
P(\inf_{t\in[0,T]}m_t\geq -\alpha, \sup_{t\in[0,T]}m_t\geq \varkappa\alpha)\leq \frac{1}{\varkappa+1}.
\end{equation*}
\end{lemma}
\begin{proof}
Without loss of generality, we can assume that our probability space can support a Wiener process  $B$ for which $B_{\langle m\rangle_t}=m_t$. Then, defining $\tau$ to be the first exit time of $B$ from  the set $(-\alpha,\varkappa\alpha)$, we have
$$
P(\inf_{t\in[0,T]}m_t\geq -\alpha, \sup_{t\in [0,T]} m_t\geq \varkappa\alpha)\leq P(B_{\tau}=\varkappa\alpha)=\frac{1}{\varkappa+1},
$$
where the equality follows from a simple application of the optional stopping theorem.
\end{proof}
\begin{lemma} \label{lem: martingale}
For any $c>0$, any continuous local martingale $m_t$ starting from 0, and any $N>0$, 
$$
P\left(\sup_{t \geq 0} (m_t-c \langle m \rangle_t )> N \right) \leq e^{-2Nc}.
$$
\end{lemma}
\begin{proof}
As before, it is not a loss of generality to assume $m_t=B_t$, where $B$ is a Wiener process. By Part II, 2.0.2.(1) in \cite{BS}, $\sup_{t \geq 0} (B_t-c t )$ has exponential distribution with parameter $2c$, which proves the claim.
\end{proof}

\begin{lemma}\label{lemma:bound from below}
Suppose that $u\in L_2([0,4],H^1(B_2))\cap L_{\infty}([0,4], L_2(B_2))$. Let  $J \subset [0,4]$ be a subinterval, $Q=B_\rho$ for some $0< \rho< 2$, $\varphi\in C^{\infty}_c(Q)$, and $\alpha>\beta\geq0$. Then
$$
\|(u-\alpha)^+\varphi\|_{2,J \times Q} ^2
\leq C(|\varphi|_{\infty}^2+|\nabla\varphi|_{\infty}^2)\left[\frac{\|(u-\beta)^+\|_{2,J\times Q} }{\alpha-\beta}\right]^{\frac{4}{d+2}}
$$
$$
\times \left[\sup_{t\in J}|(u-\alpha)^+|_{2,Q}^2+\|I_{u>\alpha}\nabla u\|_{2,J \times Q} ^2\right],
$$
with $C=C(d)$.
\end{lemma}
\begin{proof}By H\"older's inequality,
$$
\|(u-\alpha)^+\varphi\|_{2,J \times Q} ^2\leq\|(u-\alpha)^+\varphi\|_{2(d+2)/d,J \times Q}^2\|I_{u>\alpha}\|_{2,J \times Q} ^{4/d+2}.
$$
Noticing that
$$
I_{u>\alpha}\leq\frac{(u-\beta)^+}{\alpha-\beta},
$$
and using the embedding inequality 
$$
\|v\|_{2(d+2)/d,G}^2\leq C(d)\left(\sup_{t\in [0,4]}|v|_{2,B_2}^2+\|\nabla v\|_{2,G}^2\right)
$$
for $v\in L_2([0,4],H^1_0(B_2))\cap L_{\infty}([0,4], L_2(B_2))$ (see, e.g. Lemma 3.2, \cite{MOS1}), applied to the function $(u-\alpha)^+I_J\varphi$,  we get the required inequality. 
\end{proof}

Finally, let us formulate the version of It\^o's formula we will use later. We denote by $\mathcal{D}$ the set of twice continuously differentiable functions $f$ from $\R$ to $\R$, such that $f''$ is bounded. Notice that if $f \in \mathcal{D}$, then there exists a constant $\hat{K}$ such that for all $r \in \R$
$$
|f(r)| \leq \hat{K}(1+|r|^2), \ |f'(r)| \leq \hat{K}(1+|r|).
$$

\begin{lemma}                    \label{lem: Ito formula L2}
Let $u$ satisfy \eqref{SPDE}, and let $g \in \mathcal{D}$, $\varphi \in C^\infty_c(B_2)$, and $\psi \in C^{\infty}[0,4]$. Then almost surely,
$$
\int_{B_2} \varphi^2 \psi^2_t g(u_{t}) dx =\int_{B_2} \varphi^2 \psi^2_0 g(u_0) dx+\int_0^{t}\int_{B_2} 2\psi_s \psi'_s  \varphi^2 g(u_s) dx ds
$$
$$
-\int_0^{t} \int_{B_2} 2\psi^2_s \varphi\D_j\varphi g'(u_s) a^{ij}_s \D_i u_s  dx ds + \int_0^{t} \int_{B_2} \psi^2_s \varphi^2 g'(u_s) \sigma^{ik}\D_iu_s dw^k_s
$$
\begin{equation}                    \label{eq:Ito formula}
-\int_0^{t} \int_{B_2}  \psi^2_s \varphi^2 g''(u_s)[ a^{ij}_s \D_iu_s\D_ju_s - \frac{1}{2}\sigma^{ik}\sigma^{jk} \D_iu_s\D_ju_s] dx ds,
\end{equation}
for all $t \in [0,4] $.

\end{lemma}

\begin{proof}
Let $\kappa$ be nonnegative a $C^{\infty}$ function on $\R^d$, bounded by $1$, supported on $\{|x|<1\}$, and having unit integral. 
We denote $\kappa_\varepsilon(x)=\varepsilon^{-d}\kappa(x/\varepsilon)$, for $\varepsilon>0$ and for $v \in L_2(B_2)$ we write
$$
v^\varepsilon(x)=(v)^\varepsilon(x)=\int_{B_2} \kappa_\varepsilon( x-y) v(y) \,dy, \  \text{for} \ x\in\bR^d.
$$

Let us choose $\varepsilon>0$ small enough such that $\varphi$ is supported in $B_{2-\varepsilon}$. Then for $x \in B_{2-\varepsilon}$ we have 
$$
u^\varepsilon_{t}(x)=u^\varepsilon_0(x)+\int_0^{t} (a^{ij}_s\D_j u_s, \D_i\kappa_\varepsilon(x- \cdot )) dt+\int_0^{t} ( \sigma^{ik}_s\D_i u_s)^\varepsilon(x) dw^k_s.
$$
Then one can write It\^o's formula for the processes $\varphi^2(x)\psi_t^2g(u^\varepsilon_t(x))$ for $x\in B_{2}$, use Fubini and stochastic Fubini theorems (for the latter, see  \cite{KrylovITO}), and integrate by parts to obtain that almost surely, 
$$
\int_{B_2} \varphi^2 \psi^2_{t} g(u_{t}^{\varepsilon}) dx =\int_{B_2} \varphi^2 \psi_0 g(u^{\varepsilon}_0) dx+\int_0^t\int_{B_2} 2\psi_s \psi'_s  \varphi^2  g(u^{\varepsilon}_s) dx ds
$$
$$
-\int_0^t \int_{B_2} 2\psi^2_s \varphi\D_j\varphi g'(u^{\varepsilon}_s) ( a^{ij}_s \D_i u_s )^\varepsilon dx ds + \int_0^t \int_{B_2} \psi^2_s \varphi^2 g'(u^{\varepsilon}_s)( \sigma^{ik}_s\D_i u_s)^\varepsilon dxdw^k_s
$$
$$
-\int_0^t \int_{B_2} \psi^2_s \varphi^2 g''(u^{\varepsilon}_s)[ ( a^{ij}_s \D_ju_s)^\varepsilon \D_iu^{\varepsilon}_s - \frac{1}{2}( \sigma^{ik}_s\D_i u_s)^\varepsilon( \sigma^{jk}_s\D_j u_s)^\varepsilon  ]dx ds,
$$
for all $t\in[0,4] $. Then for fixed $t$ one lets $\varepsilon \to 0$ to obtain that \eqref{eq:Ito formula} holds almost surely, and the result follows since both sides of \eqref{eq:Ito formula} are continuous in $t$.
\end{proof}

\begin{lemma}\label{lem:Ito}
Let $u$ satisfy \eqref{SPDE}, and let $g \in \mathcal{C}$,  $\varphi \in C^\infty_c(B_2)$, and $\psi \in C^{\infty}[0,4]$.  Set $v_t= (g(u_t))^+$. Then $ v \in \mathcal{H}$, and almost surely,
\begin{align}
\nonumber
|\varphi \psi_{t}  v_t |_2^2 &=|\varphi \psi_0 v_0|_2^2+ \int_0^{t} \int_{B_2} 2\psi^2_s \varphi^2 v_s  \sigma^{ik}\D_iv_s dx dw^k_s +\int_0^{t}\int_{B_2}2 \psi_s \psi'_s \varphi^2  v_s ^2 dx ds \\ \nonumber
&-\int_0^{t} \int_{B_2}  \psi^2_s \varphi^2 [ 2a^{ij}_s \D_iv_s\D_jv_s - \sigma^{ik}\sigma^{jk} \D_iv_s\D_jv_s] dx ds \\ \nonumber
& -\int_0^{t} \int_{B_2}  \psi^2_s \varphi^2 v_sg''(u_s)[ 2a^{ij}_s \D_iu_s\D_ju_s - \sigma^{ik}\sigma^{jk} \D_iu_s\D_ju_s] dx ds \\            \label{eq:Ito positive part}
&-\int_0^{t} \int_{B_2} 4\psi^2_s \varphi\D_j\varphi v_s  a^{ij}_s \D_i v_s   dx ds,
\end{align} 
for all $t \in [0,4] $. 
\end{lemma}
\begin{proof}
Since $g$ has bounded first derivative, it follows easily that $v \in \mathcal{H}$. We introduce now the functions $\alpha_\delta(r)$, 
$\beta_\delta(r)$ and $\gamma_\delta(r)$ on 
$\mathbb{R}$, for  $\delta>0$,  given by
\begin{equation}
\alpha_ \delta (r)= \left\{
\begin{array}{rl}
2 & \text{if } r > \delta \\
\frac{2r}{\delta} & \text{if } 0 \leq r \leq \delta \\
0 & \text{if } r < 0,
\end{array} \right. \nonumber
\end{equation}
\[ \beta _\delta (r) = \int_0^r a_\delta (s)ds,
\qquad \gamma_\delta(r)=\int_0^r \beta_\delta(s) ds.\]
For all $r \in \mathbb{R}$ we have $\alpha_\delta(r) \to 2I_{r >0}$, 
$ \beta_\delta(r) \to 2r^+$ 
and $\gamma_\delta(r) \to {(r^+)^2}$ as $\delta \to 0$. 
Also, for all $r,r_1,r_2$ and $\delta$,  
the following inequalities hold
\begin{equation*}                  \label{eq:growth approximation}
|\alpha_\delta(r)|\leq 2, \ |\beta_\delta(r)| 
\leq2 |r|, \ |\gamma_\delta(r)| \leq {r^2}.
\end{equation*}
It follow then that since $g \in \mathcal{C}$, the function $\zeta_\delta(r):=\gamma_\delta ( g (r))$ lies in $\mathcal{D}$. Hence, by virtue of Lemma \ref{lem: Ito formula L2} one can write  It\^o's formula for $|\psi \varphi \sqrt{\zeta}_\delta(u_t)|^2_2$, i.e. \eqref{eq:Ito formula} with $g, g'$ and $g''$ replaced by $ \gamma_\delta(g), \beta_\delta(g)g'$ and $\alpha_\delta(g)|g'|^2+\beta_\delta(g)g''$ respectively. Then we let $\delta \to 0$ to obtain \eqref{eq:Ito positive part}.

\end{proof}

\section{Proofs of the main results}  \label{sec: Proofs}

\begin{proof}[Proof of Theorem \ref{thm: supremum estimate}]
We first prove $(i)$ It is easy to see that it suffices to show the existence of $\hat{\gamma},\delta_1, \delta_2,C>0$ such that
\begin{equation}\label{eq:thmB equiv}
P(\|f(u)^+\|_{\infty,G_1}^2\geq 1,\|f(u)^+\|_{2,G_{3/2}}^2\leq\kappa^{-\delta_1}\hat{\gamma})\leq C\kappa^{-\delta_2},
\end{equation}
since by substituting $\tilde{f}= f(\gamma /\kappa^{\delta_1} \alpha)^{1/2}$ in place of $f$ in \eqref{eq:thmB equiv}, we obtain the desired   inequality with $\delta=\delta_2/ \delta_1$ and $\hat{C}=1/ \hat{\gamma}$.

Let us take $r\in[0,4]$, $\rho \in [1,2]$, $\psi\in C^{\infty}([0,4])$ with $\psi=0$ on $[0,r]$, and $\varphi\in C_c^{\infty}(B_{\rho})$.
For $j=0,1,\ldots$ let $g^j(u):=f(u)-(1-2^{-j})$, $v^j=(g^j(u))^+$, and let us apply Lemma \ref{lem:Ito} with $g^{j+1}$.  Using the parabolicity condition and Young's inequality, as well as the nonnegativity of $v^{j+1}(g^{j+1})''$, we get for any $\varepsilon >0$

$$
\int_{B_{\rho}} \varphi^2\psi_t^2|v_t^{j+1}|^2dx
$$
$$\leq m_t^{j+1}+\int_r^t\int_{B_{\rho}}2\psi_s\psi'_s \varphi^2 |v_s^{j+1}|^2dxds
-\int_r^t\int_{B_{\rho}}\lambda\psi_s^2\varphi^2|\nabla v_s^{j+1}|^2dxds
$$
$$
+\int_r^t\int_{B_{\rho}}[\varepsilon\psi_s^2\varphi^2K^2|\nabla v_s^{j+1}|^2+16/\varepsilon\psi_s^2|\nabla \varphi|^2|v_s^{j+1}|^2]dxds
$$
almost surely for all for $t\in[r,4]$, where 
$$
m_t^{j+1}= \int_r^t \int_{B_{\rho}} 2\varphi^2\psi_s^2v^{j+1}_sM^kv^{j+1}_sdxdw^k_s.
$$
Choosing $\epsilon$ sufficiently small, we arrive at
$$
\int_{B_{\rho}} \varphi^2 \psi^2_t|v_t^{j+1}|^2dx
+\int_{r}^{t}\int_{B_{\rho}}\varphi^2\psi^2|\nabla v^{j+1}_s|^2dxds
$$
\begin{equation}\label{energy0}
\leq C'm^{j+1}_t + C\int_r^t \int_{B_{\rho}}  \varphi^2\psi_s \psi_s'|v_s^{j+1}|^2+|\nabla\varphi|^2\psi_s^2|v_s^{j+1}|^2dxds.
\end{equation}
Now let us choose $r=r_j=3-(5/4)2^{-j}$ and $\rho=\rho_j=1+(1/2)2^{-j}$, that is, $[r_0,4]\times B_{\rho_0}=G_{3/2}$. Also we introduce the notation $F_j=[r_{j},4] \times B_{\rho_j}$. Furthermore, choose $\psi=\psi^j$ and $\varphi=\varphi^j$ such that
\begin{enumerate}[(i)]
\item $0\leq \psi^j\leq1$, $\psi^j|_{[0,r_j]}=0,\; \psi^j|_{[r_{j+1},4]}=1$;
\item $0 \leq \varphi^j\leq1$, $\varphi^j\in C_0^{\infty}(B_{\rho_j}),\; \varphi^j|_{B_{\rho_{j+1}}}=1$;
\item $|\D_t\psi^j|+|\nabla\varphi^j|^2< C 4^j$.
\end{enumerate}
Then by running $t$ over $[r_j,4]$, by \eqref{energy0} we obtain
\begin{equation}\label{energy1}
\sup_{t\in [r_{j+1},4]}|v_t^{j+1}|_{2,B_{\rho_{j+1}}}^2+\|\nabla v^{j+1}\|_{2,F_{j+1}}^2\leq C 4^j \|v^{j+1}\|_{2,F_j}^2+C\sup_{t\in[r_j,4]}m^{j+1}_t.
\end{equation}
Notice that, since the left-hand side of \eqref{energy0} is nonnegative, running $t$ over $I_j$ gives
\begin{equation}\label{eq:infimum}
\inf_{t\in[r_j,4]}m^{j+1}_t\geq -C4^j \|v^{j+1}\|_{2,F_j}^2.
\end{equation}
Applying Lemma \ref{lemma:bound from below} with $\alpha=1-2^{-(j+1)}$, $\beta=1-2^{-j}$, and $\varphi=\varphi^{j+1}$, we get
\begin{align*}
\|v^{j+1}\|_{2,F_{j+2}}&\leq\|\varphi^{j+1}v^{j+1}\|_{2,F_{j+1}}
\\&
\leq C^j\|v^j\|_{2,F_{j+1}}^{4/d+2}
\left[\sup_{t\in [r_{j+1},4]}|v_t^{j+1}|_{2,B_{\rho_{j+1}}}^2+\|\nabla v^{j+1}\|_{2,F_{j+1}}\right].
\end{align*}
Combining this with \eqref{energy1} yields
$$
\|v^{j+1}\|_{2,F_{j+2}}^2\leq C^j\|v^j\|_{2,F_{j+1}}^{4/d+2}\left[ \|v^{j+1}\|_{2,F_j}^2+\sup_{t\in[r_j,4]}m_t^{j+1}\right].
$$
Since for $j>i$, we have $v^j  \leq v^i$ and $F_{j} \subset F_i$,  we obtain for $V_j=\|v^j\|_{2,F_j}^2$
$$
V_{j+2}\leq C^jV_j^{2/(d+2)}\left[4^jV_j+\sup_{t\in[r_j,4]}m_t^{j+1}\right].
$$
Let $\gamma_0,\gamma\in(0,1)$ and suppose that $V_j\leq\gamma_0\gamma^j$ on a set $\Omega_j\subset\Omega$. By \eqref{eq:infimum} we have
$$
\inf_{t\in[r_j,4]}m_t^{j+1}\geq -C4^j \|v^{j+1}\|_{2,F_j}^2\geq- C4^jV_j,
$$
and Lemma \ref{lem: martingale 2} can be applied with $\alpha=C4^j \gamma_0 \gamma ^j$, $\varkappa=\kappa 4^j$. That is we obtain a subset $\Omega_{j+2}\subset\Omega_j$ such that $P(\Omega_j\setminus\Omega_{j+2})\leq  \kappa^{-1}4^{-j}$ and on $\Omega_{j+2}$
$$
\sup_{t\in[r_j,4]}m_t^{j+1}\leq \kappa C 16^j\gamma_0 \gamma^j.
$$
Consequently, on $\Omega_{j+2},$
$$
V_{j+2}\leq C^j\gamma_0^{2/(d+2)} \gamma^{2j/(d+2)} \gamma_0 \gamma^j (1+\kappa)  \leq\gamma_0\gamma^{j+2},
$$
provided that
$$
\gamma=C^{-(d+2)/2},\;\gamma_0\leq(\gamma^2/(1+\kappa))^{(d+2)/2}.
$$
Proceeding iteratively, we can conclude that on $\cap_{j\geq0}\Omega_{2j}$, $V_j\rightarrow0$, and therefore 
$$\|f(u)^+\|_{\infty,G_1}^2\leq 1,$$
and moreover,
$$
P\left(\Omega_0\setminus\cap_{j\geq0}\Omega_{2j}\right)=\sum_{j\geq0}P(\Omega_{2j}\setminus\Omega_{2j+2})
\leq2C\kappa^{-1}.
$$
This proves \eqref{eq:thmB equiv}, since $\Omega_0=\{\|f(u)^+\|_{2,G_{3/2}}^2=V_0\leq\gamma_0=\kappa^{-(d+2)/2}\hat{\gamma}\}$ with a constant $\hat{\gamma}=\hat{\gamma}(d,\lambda,K)$.

For part $(ii)$, we have 
$$
P(\|f(u)\|_{\infty,G_1}^2\geq\hat{C}\kappa\alpha,\|f(u)\|_{2,G_{3/2}}^2\leq\alpha)
$$
$$
\leq P(\|f(u)^+\|_{\infty,G_1}^2\geq\hat{C}\kappa\alpha,\|f(u)^+\|_{2,G_{3/2}}^2\leq\alpha)
$$
$$
+P(\|f(u)^-\|_{\infty,G_1}^2\geq\hat{C}\kappa\alpha,\|f(u)^-\|_{2,G_{3/2}}^2\leq\alpha),
$$
which by virtue of $(i)$ and the fact that $-f$ satisfies the conditions of the lemma, yields (ii).

Note that in case the initial value of $f(u)$ is identically 0, the time-cutoff function $\psi$ in the above argument can be omitted. Doing so and repeating the same steps leads to a proof of (iii)-(iv). 
\end{proof}

Recall that from \cite{ROZ} it is known that solutions of \eqref{SPDE} with 0 boundary and $L_p$ initial conditions are weakly continuous in $L_p$ for any $p\in(0,\infty)$. A simple consequence of Theorem \ref{thm: supremum estimate} is that in fact strong continuity holds, away from $t=0$.

\begin{corollary}\label{cor}
Let $u$ be a solution of \eqref{SPDE} and $p\in(0,\infty)$. Then
\begin{enumerate}[(i)]
\item  $(u_t)_{t\in[3,4]}$ is strongly continuous in $L_p(B_1);$
\item If furthermore $u|_{\partial B_2}=0$, then $(u_t)_{t\in[3,4]}$ is strongly continuous in $L_p(B_2)$.
\end{enumerate}
\end{corollary}
\begin{proof}
$(i)$  First notice that the supremum in time can be taken to be real (and not only essential) supremum: the function $|(u-\|u\|_{\infty,G_1})^+|_{2,B_1}$ is 0 for almost all $t$, hence by the continuity of $u$ in $L_2$ it is 0 for all $t$, and therefore, for all $t$, almost every $x$, $u_t(x)\leq\|u\|_{\infty,G_1}$. Now fix $t\in[3,4]$, and take a sequence $t_n\rightarrow t$. Then $u_{t_n}\rightarrow u_{t}$ in $L_2(B_2)$, hence for a subsequence $t_{n_k}$, for almost every $x$. On the other hand, $|u_{t_{n_k}}I_{B_1}|\leq \|u\|_{\infty,G_1}<\infty$, therefore by Lebesgue's theorem, $u_{t_{n_k}}\rightarrow u_t$ in $L_p$.

For part $(ii)$, notice that when $u\in H^1_0(B_2)$ for almost all $\omega,t$, then in the special case $f(r)=r$ the space-cutoff function $\varphi$ in the proof of Theorem \ref{thm: supremum estimate} can be omitted. We then obtain that $\|u\|_{\infty,[3,4]\times B_2}<\infty$ with probability 1, and by the same argument as above we get the claim.
\end{proof}

Before turning to the proof of Theorem \ref{thm: harnack}, we need one more lemma, which can be considered as a weak version of Theorem \ref{thm: harnack}.

\begin{lemma}\label{lem: weakweak}
Let $u$ be a solution of \eqref{SPDE}, such that on $A \in \mathscr{F}$, $u\in \Lambda_{\eta}$. Then  for any $N>0$,  there exists a set $D_1 \in \mathscr{F}$, with $P(D_1) \leq  C e^{-cN}$, such that on $A \cap D_1^c$, for all $t\in [0,4],$ 
$$
|\{(x\in B_\rho|\,u(t,x)\geq e^{-N}\}|\geq \eta^3|B_\rho|,
$$
where $\rho$ is defined by
\begin{equation*}\label{rho}
|B_\rho|=(1/(1+\eta)\vee3/4)|B_2|,
\end{equation*}
and the constants $c,C>0$ depend only on $d$, $\lambda$, $K$, and $\eta$.
\end{lemma}

\begin{proof} Clearly it is sufficient to prove the statement for $N>N_0$ for some $N_0$. Introduce the functions
\begin{equation}
f_h(x)= \left\{
\begin{array}{rl}
a_hx+b_h & \text{if } x< -h/2 \\
\log^+\frac{1}{x+h} & \text{if } x\geq-h/2,
\end{array} \right. \nonumber
\end{equation}
for $h>0$ where $a_h$ and $b_h$ is chosen such that $f_h$ and $f'_h$ are continuous. Let $\kappa$ be nonnegative a $C^{\infty}$ function on $\R$, bounded by $1$, supported on $\{|x|<1\}$, and having unit integral. Denote $\kappa_h(x)=h^{-1}\kappa(x/h)$ and
$$F_h=f_h\ast\kappa_{h/4}.$$
We  claim that $F_h$ has the following properties:
\begin{enumerate}[(i)]
\item $F_h(x)=0$ for $x\geq1$;
\item $F_h(x)\leq \log(2/h)$ for $x\geq 0$;
\item $F_h(x)\geq \log(1/2h)$ for $x\leq h/2$;
\item $F_h\in\mathcal{D}$ and $F''_h(x)\geq (F'_h(x))^2$ for $x\geq0$.
\end{enumerate}
The first three properties are obvious, while for the last one notice that $F_h$ has bounded second derivative, $f''_h(x)\geq (f'_h(x))^2$ for $x\geq h/2$ and $x\neq 1-h$,, and therefore, for $x\geq 0$
\begin{align*}
(F_h'(x))^2&=\left(\int f'_h(x-z)\kappa^{1/2}_{h/4}(z)\kappa^{1/2}_{h/4}(z)\,dz\right)^2\\
&\leq \int (f'_h(x-z))^2\kappa_{h/4}(z)\,dz\\
&\leq \int f_h''(x-z)\kappa_{h/4}(z)dz\\
& \ \  \ \ +[\lim_{y\rightarrow (1-h)^+}f_h'(y)-\lim_{y\rightarrow (1-h)^-}f_h'(y)]\kappa_{h/4}(x-1+h)=F_h''(x),
\end{align*}
where the integrals are understood in the usual Lebesgue sense and not as a formal expression for the action of distributions. Let us denote $v=F_h(u)$. Applying Lemma \ref{lem: Ito formula L2} and using the parabolicity condition, we get
$$
\int_{B_2}\varphi^2v_t\,dx-\int_{B_2}\varphi^2v_0\,dx\leq\int_0^t\int_{B_2} C\varphi\nabla\varphi\nabla v
-(\lambda/2) \varphi^2F_h''(u)(\nabla u)^2\,dx\,ds
$$
\begin{equation}\label{eq:intermediate0}
+\int_0^t\int_{B_2}\varphi^2M^kv\,dx\,dw^k_s
\end{equation}
for any $\varphi\in C_c^{\infty}.$ Let us denote the stochastic integral above by $m_t$, and notice that provided $|\varphi|\leq1$, 
$$
\langle m\rangle_t\leq C\int_0^t\int_{B_2}\varphi^2(\nabla v)^2\,dx\,ds.
$$
Let $c$ be such that $cC\leq\lambda/4$. From Lemma \ref{lem: martingale}, there exists a set $D_1$ with $P(D_1) \leq  e^{-2Nc}$, such that on $D_1^c$ we have

$$\int_{B_2}\varphi^2v_t\,dx-\int_{B_2}\varphi^2v_0\,dx$$
\begin{equation}\label{eq: intermediate 2}
\leq N+\int_0^t\int_{B_2} C\varphi\nabla\varphi\nabla v-(\lambda/2) \varphi^2F_h''(u)(\nabla u)^2+cC\varphi^2(\nabla v)^2\,dx\,ds.
\end{equation}
On $A \cap D_1^c$, by the property (iv) above, we have $F_h''(u)(\nabla u)^2\geq(\nabla v)^2$, and therefore
\begin{equation}\label{eq:intermediate1}
\int_{B_2} \varphi^2v_t\,dx\leq N+C\int_{B_2}|\nabla\varphi|^2\,dx+\int_{B_2} \varphi^2v_0\,dx.
\end{equation}
Let us denote 
$$
\mathcal{O}_t(h)=\{x\in B_{\rho}:u(t,x)\geq h\}.
$$
Choosing $\varphi$ to be $1$ on $B_{\rho}$, by properties (i), (ii), and (iii) of $F_h$ and \eqref{eq:intermediate1}, on $A \cap D_1^c$, for all $t\in[0, 4]$
\begin{align*}
|B_{\rho}\setminus\mathcal{O}_t(h/2)|\log(1/2h)&\leq C+N+(1-\eta)\log(2/h)|B_2|\\& \leq C+N+(1-\eta^2)\log(2/h)|B_{\rho}|.
\end{align*}
Hence
\begin{equation*}             \label{eq: before choosing h}
|\mathcal{O}_t(h/2)| \geq |B_\rho|-\frac{C+N}{\log(1/2h)}
-(1-\eta^2)\frac{\log(2/h)}{\log(1/2h)} |B_\rho|,
\end{equation*}
and choosing $N_0=C$ and $h=2e^{-C'N}$ for a sufficiently large $C'$ finishes the proof of the lemma.
\end{proof}

\begin{proof}[Proof of Theorem \ref{thm: harnack}]
By Lemma \ref{lem: weakweak}, there exists a set $D_1$ with $P(D_1) \leq Ce^{-cN}$ such that on $A \cap D_1^c$ we have
\begin{equation}                     \label{eq: estimate of the measure}
|\{(x\in B_\rho|\,u(t,x)\geq e^{-N}|\geq \eta^3|B_\rho|,
\end{equation}
for all $t \in [0,4]$. Let us denote $h:=e^{-N}$. For $0<\epsilon \leq h/2$, we introduce the function
\begin{equation}
f_\epsilon(x)= \left\{
\begin{array}{rl}
a_\epsilon x+b_\epsilon & \text{if } x< -\epsilon/2 \\
\log^+\frac{h}{x+\epsilon} & \text{if } x\geq-\epsilon/2,
\end{array} \right. \nonumber
\end{equation}
where $a_\epsilon$ and $b_\epsilon$ is chosen such that $f_\epsilon$ and $f'_\epsilon$ are continuous. Let $\kappa$ be a nonnegative  $C^{\infty}$ function on $\R$, bounded by $1$, supported on $\{|x|<1\}$, and having unit integral. Denote $\kappa_\varepsilon(x)=\varepsilon^{-1}\kappa(x/\varepsilon)$ and
$$F_\epsilon=f_\epsilon\ast\kappa_{\epsilon/4}.$$
Similarly to $F_h$ in the proof of Lemma \ref{lem: weakweak}, $F_\epsilon$ has the following properties:
\begin{enumerate}[(i)]
\item $F_\epsilon(x)=0$ for $x\geq h$;
\item $F_\epsilon(x)\leq\log(2h/\epsilon)$ for $x\geq 0$;
\item $F_\epsilon(x)\geq\log(h/(x+\epsilon))-1$ for $x\geq 0$;
\item $F_\epsilon\in\mathcal{D}$ and $F''_\epsilon(x)\geq (F'_\epsilon(x))^2$ for $x\geq0$.
\end{enumerate}
Let us denote $v=F_{\epsilon}(u)$. Similarly to \eqref{eq: intermediate 2}, there exists a set $D_2$ with $P(D_2) \leq C e^{-Nc}$, such that on $D_2^c$ we have
$$
\int_{B_2}\varphi^2v_t\,dx-\int_{B_2}\varphi^2v_0\,dx$$
$$
\leq N+\int_{0}^t\int_{B_2} C\varphi\nabla\varphi\nabla v-(\lambda/2) \varphi^2F_{\epsilon}''(u)(\nabla u)^2+(\lambda/4)\varphi^2(\nabla v)^2\,dx\,ds.
$$
On $A \cap D_2^c$, by property (iv), we have, 
\begin{equation}
\int_{0}^4 \int_{B_2} \varphi^2|\nabla v_t|^2 \,dxdt\leq C (N+\int_{B_2} |\nabla\varphi|^2\,dx+\int_{B_2} \varphi^2v_2\,dx).
\end{equation}
By choosing $\varphi \in C^\infty_c({B_2})$ with $0 \leq \varphi \leq 1$ and $\varphi={1}$ on $B_\rho$ we get,
$$
\int_{0}^4 \int_{B_\rho}|\nabla v_t|^2 \,dxdt\leq C (N+\int_{B_2} |\nabla\varphi|^2\,dx+\int_{B_2} \varphi^2v_0\,dx).
$$
Hence, by property (ii),
\begin{equation}                     \label{eq: before poincare's}
\int_{0}^4 \int_{B_\rho}|\nabla v_t|^2 \,dxdt\leq C N+ C+C\log\frac{2h}{\epsilon}.
\end{equation}
Using property (i), by a version of Poincar\'e's inequality (see, e.g., Lemma II.5.1, \cite{LA}) we get for all $t$
$$
\int_{B_\rho}|v_t|^2dx \leq C\frac{\rho^{2(d+1)}}{|\mathcal{O}_t(h)|^2} \int_{B_\rho}| \nabla v_t|^2 dx ,
$$
which, by virtue of \eqref{eq: estimate of the measure} and \eqref{eq: before poincare's} implies 
$$
\int_{0} ^4 \int_{B_\rho}|v_t|^2dx \leq C +CN+C\log\frac{2h}{\epsilon}.
$$
on $A\cap D_1^c\cap D_2^c$. By Theorem \ref{thm: supremum estimate} and noting that $G_{3/2}\subset[0,4]\times B_{\rho}$ for any $\rho=\rho(d,\eta)$ as defined in Lemma \ref{lem: weakweak}, we get that there exists a set $D_3\in \mathscr{F}$ with $P(D_3) \leq C N^{-\delta}$, such that on $A\cap D_1^c\cap D_2^c\cap D_3^c$ we have
$$
\sup_{(t,x) \in {G_1}} v_t(x) \leq  [N(C +CN+C\log\frac{2h}{\epsilon})]^{1/2}.
$$
By applying property (iii), we get
$$
\sup_{(t,x) \in {G_1}} \log \frac{h}{ u_t(x)+\epsilon} \leq [N (C +CN+C\log\frac{2h}{\epsilon})]^{1/2}+1,
$$
and therefore,
$$
\inf_{(t,x) \in {G_1}}  u_t(x) \geq he^{-[N(C+CN+C\log 2h -C\log \epsilon)]^{1/2}-1}- \epsilon.
$$      
Letting $\epsilon=e^{-c'N}$ with a sufficiently large $c'$, it is easy to see that the right-hand side above is bounded from below by $\epsilon$, finishing the proof.
\end{proof}
\vspace{0,5 cm}
In the following proof, whenever we refer  to Theorem \ref{thm: harnack}, we mean the particular case $\eta =1/2$. 
\begin{proof}[Proof of Theorem \ref{thm: continuity}]
  Consider the parabolic transformations $\mathfrak{P}_{\alpha,t',x'}$:
$$
t\rightarrow\alpha^2t+t',
$$
$$
x\rightarrow\alpha x+x'.
$$
It is easy to see that if $v$ is a solution of \eqref{SPDE} on a cylinder $Q$, then $v\circ\mathfrak{P}_{\alpha,t',x'}^{-1}$ is also solution of \eqref{SPDE}, on the cylinder $\mathfrak{P}_{\alpha,t',x'}Q$, with another sequence of Wiener martingales on another filtration, and with different coefficients that still satisfy Assumption \ref{as0} with the same bounds. To ease notation, for a cylinder $Q$ let $\mathfrak{P}_{Q}$ denote the unique parabolic transformation that maps $Q$ to $G$, if such exists. Also, for an interval $[s,r]\subset[0,4]$ let $\mathfrak{P}_{[s,r]}=\mathfrak{P}_{2/\sqrt{r-s},-4s/(r-s),0}$. That is, $\mathfrak{P}_{[s,r]}[s,r]\times B_1=[0,4]\times B_{2/\sqrt{r-s}}$, which, when $r-s\leq1$, contains $G$. 

Without loss of generality $x_0=0$ can and will be assumed, as will the almost sure boundedness of $u$ on $G$, since these can be achieved with appropriate parabolic transformations, using the boundedness obtained on sub-cylinders in Theorem \ref{thm: supremum estimate}. Also let us fix a probability $\delta>0$, denote the corresponding lower bound $3\epsilon_2$ obtained from the Harnack inequality, and take an arbitrary $0<\epsilon_1<\epsilon_2/2$.

Apply Theorem \ref{thm: supremum estimate} $(iii)$ twice,  with the function $f(r)=r$, with the interval $[t_0-4s,t_0+s]$, and with solutions $v=u-\sup_{\{t_0-4s\}\times B_2}u$ and $v=-u+\inf_{\{t_0-4s\}\times B_2}u$. Also notice that (for both choices of $v$) 
$$
\|v^+\|^2_{2,[t_0-4s,t_0+s]\times B_2}\leq Cs\|u\|^2_{\infty,G}\rightarrow0
$$
as $s\rightarrow0$ for almost every $\omega$, and thus in probability as well (recall that the fact that  the functions $v$ are well-defined and that the above - seemingly trivial - inequality holds, is justified in the proof of Corollary \ref{cor}). In other words,
$$
P(\|v^+\|^2_{2,[t_0-4s,t_0+s]\times B_2}>\alpha)
$$
can be made arbitrarily small by choosing $s$ sufficiently small. Therefore, we obtain an $s>0$ and an event $\Omega_0$, with $P(\Omega_0)>1-\delta$, such that on $\Omega_0$, 
$$
\sup_{[t_0-4s,t_0+s]\times B_1}u-\sup_{\{t_0-4s\}\times B_2}u<\epsilon_1^2/6
$$
$$\inf_{[t_0-4s,t_0+s]\times B_1}u-\inf_{\{t_0-4s\}\times B_2}u>-\epsilon_1^2/6.
$$
Let us rescale $u$ at the starting time:
$$
u'_{\pm}(t,x)=\pm\left(2\frac{u(t,x)-\sup_{\{t_0-4s\}\times B_2} u}{\sup_{\{t_0-4s\}\times B_2} u-\inf_{\{t_0-4s\}\times B_2} u}+1\right),
$$
that is, $\sup_{B_2} u'_{\pm}(t_0-4s,\cdot)= 1, \inf_{B_2} u'_{\pm}(t_0-4s,\cdot)=-1$. Now we can write $\Omega_0=\Omega_A\cup\Omega_B$, where 
\begin{itemize}
\item On $\Omega_A$, $\osc_{\{t_0-4s\}\times B_2}u<\epsilon_1/3$, and therefore, $\osc_{[t_0-4s,t_0+s]\times B_1}u<\epsilon_1/3+2\epsilon_1^2/6<\epsilon_1$;
\item On $\Omega_B$, $|u'_{\pm}|<1+2(\epsilon_1^2/6)/(\epsilon_1/3)=1+\epsilon_1$, on $[t_0-4s,t_0+s]\times B_1$.
\end{itemize}
Notice that in the event $\Omega_B$, on the cylinder $[t_0-4s,t_0+s]\times B_1$, the functions $u'_{\pm}/(1+\epsilon_1)+1$ take values between 0 and 2. Therefore one of $(u'_{\pm}/(1+\epsilon_1)+1)\circ\mathfrak{P}_{[t_0-4s,t_0+s]}^{-1}\Big|_G$ (see  Figure \ref{figure} below ), denoted for the moment by $u''$, satisfies the conditions of Theorem \ref{thm: harnack} with $A=\Omega_B$.

 We obtain that on an event $\Omega_B'$
$$
\inf_{G_1} u''>3\epsilon_2,
$$
and thus 
$$
\osc_{Q}u<\frac{(2-3\epsilon_2)(1+\epsilon_1)}{2}\osc_{\{t_0-4s\}\times B_2}u<(1-\epsilon_2)\osc_{\{t_0-4s\}\times B_2}u,
$$
where $Q=\mathfrak{P}_{[t_0-4s,t_0+s]}^{-1}G_1$. Moreover, $P(\Omega_B\setminus\Omega'_B)<\delta$. Also, notice that $(t_0,0)\in Q.$  Let us denote $\Omega_1=\Omega_A\cup\Omega'_B$.
We have shown the following lemma:
\begin{figure}[t]              \label{Fig}
\centering
\includegraphics[scale=0.5]{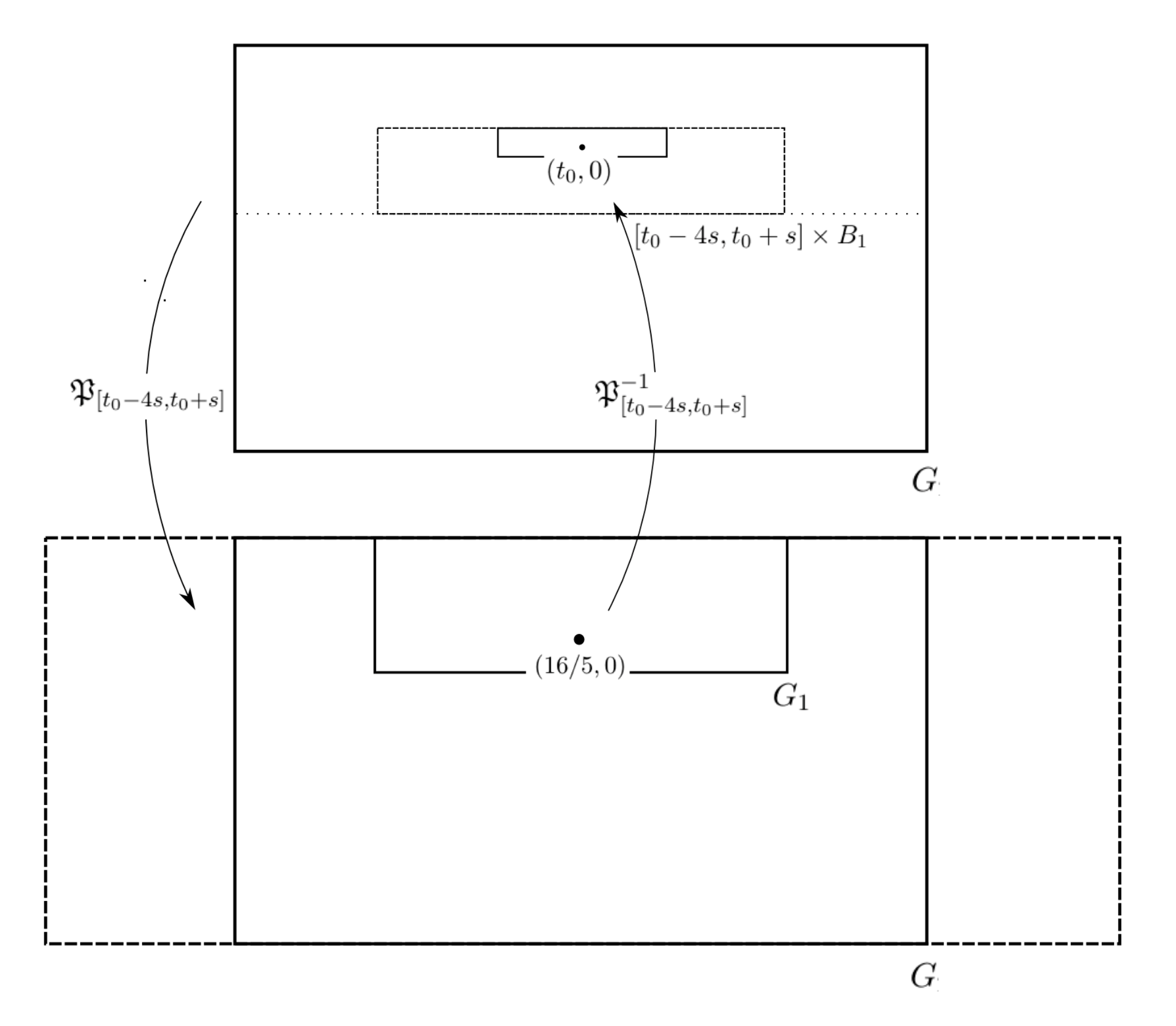}
\caption{The cylinders appearing in one iteration step}
\label{figure}
\end{figure}

\begin{lemma}\label{lem: intermediate}
Let $\delta>0$ and let $3\epsilon_2$ be the lower bound corresponding to the probability $\delta$ obtained from the Harnack inequality. For any $u$ that is a solution of \eqref{SPDE} on $G$, $t_0>0$, and for any sufficiently small $\epsilon_1>0$ there exists an $s>0$ and an event $\Omega_1$ such that

\begin{enumerate}[(i)]
\item $P(\Omega_1)>1-2\delta$;
\item On $\Omega_1$, at least one of the following is satisfied:
\begin{enumerate}[(a)]
\item $\osc_{Q}u<\epsilon_1$;
\item $\osc_{Q}u<(1-\epsilon_2) \osc_{G}u$,
\end{enumerate}
\end{enumerate}
where $Q=\mathfrak{P}_{[t_0-4s,t_0+s]}^{-1}(G_1)$.
\end{lemma}
Now take $u=u^{(0)}$ and $t_0=t_0^{(0)}$ from the statement of the theorem and a sequence $(\epsilon_1^{(n)})_{n=0}^{\infty}\downarrow0$, and for $n\geq0$ proceed inductively as follows:
\begin{itemize}
\item Apply Lemma \ref{lem: intermediate} with $u^{(n)}$, $t_0^{(n)}$, and $\epsilon_1^{(n)}$, and take the resulting $\Omega_1^{(n)}$ and $Q^{(n)}$;
\item Let $u^{(n+1)}=u^{(n)}\circ\mathfrak{P}_{Q^{(n)}}^{-1}$ and $(t_0^{(n+1)},0)=\mathfrak{P}_{Q^{(n)}}(t_0^{(n)},0)$. 
\end{itemize}

On $\limsup_{n\rightarrow\infty}\Omega_1^{(n)}$ the function $u$ is continuous at the point $(t_0,0)$. Indeed, the sequence of cylinders $Q^{(0)},\mathfrak{P}_{Q^{(0)}}^{-1}Q^{(1)},\mathfrak{P}_{Q^{(0)}}^{-1}\mathfrak{P}_{Q^{(1)}}^{-1}Q^{(2)},\ldots$ contain $(t_0,0)$, and the oscillation of $u$ on these cylinders tends to 0. However, we have $P(\limsup_{n\rightarrow\infty}\Omega_1^{(n)})\geq 1-2\delta$, and since $\delta$ can be chosen arbitrarily small, $u$ is continuous at $(t_0,0)$ with probability 1, and the proof is finished.

\end{proof}

\begin{remark}It is natural to attempt to modify the above argument to bound expectations and higher moments of the oscillations, in the hope to apply Kolmogorov's continuity criterion and obtain H\"older estimates. A main obstacle appears to be to establish a uniform integrability property to a family of (normalized) oscillations. Indeed, the present Harnack inequality can bring down the oscillation by a given factor outside of a small event, and therefore one would like to exclude the possibility that the majority of the oscillation's mass is concentrated on that exceptional event.
\end{remark}

\section*{Acknowledgement}
The authors are grateful towards  Istv\'an Gy\"ongy for the fruitful discussions during the preparation of this paper. 

% AOS,AOAS: If there are supplements please fill:
%\begin{supplement}[id=suppA]
%  \sname{Supplement A}
%  \stitle{Title}
%  \slink[doi]{10.1214/00-AOASXXXXSUPP}
%  \sdatatype{.pdf}" 
%  \sdescription{Some text}
%\end{supplement}




\begin{thebibliography}{100}

\bibitem{BS}A. N. Borodin, P. Salminen, Handbook of Brownian Motion - facts and formulae. 2nd edn. Birkhauser Verlag, Basel, 2002, xv+658 pp.

\bibitem{CAR}L. Caffarelli, L. Silvestre, 
Regularity theory for fully nonlinear integro-differential equations. 
Comm. Pure Appl. Math. 62 (2009), no. 5, 597-638. 

\bibitem{Chen}G. Chen,  Non-divergence Parabolic Equations of Second Order with Critical Drift in Lebesgue Spaces, arXiv:1511.01215.

\bibitem{MAKO}K. Dareiotis, M. Gerencs\'er, On the boundedness of solutions of SPDEs. Stoch. Partial Differ. Equ. Anal. Comput. 3 (2015), no. 1, 84-102.

\bibitem{DDH} A. Debussche, S. De Moor, M. Hofmanova, A regularity result for quasilinear stochastic partial differential equations of parabolic type, SIAM J. Math. Anal. 47 (2015), no. 2, 1590-1614.


\bibitem{DG} E. De Giorgi, Sulla differenziabilit\`a e l'analiticit\`a delle estremali degli integrali multipli regolari, Mem. Accad. Sci. Torino. Cl. Sci. Fis. Math. Nat., 3,  1957, 25-43.

\bibitem{DIBE} E. DiBenedetto, U. Gianazza, V. Vespri ; Harnack's inequality for degenerate and singular parabolic equations. Springer Monographs in Mathematics. Springer, New York, 2012. xiv+278 pp.

\bibitem{Har} A. Harnack, Die Grundlagen der Theorie des logarithmischen Potentiales und der eindeutigen Potentialfunktion in der Ebene, V. G. Teubner, Leipzig, 1887.

\bibitem{HWW} E. P. Hsu, Y. Wang, Z. Wang, Stochastic De Giorgi Iteration and Regularity of Stochastic Partial Differential Equation, Ann. Prob, to appear.

\bibitem{Kim} K. Kim, $L_p$-Estimates for SPDE with Discontinuous Coefficients in Domains, Electron. J. Probab. 10 (2005), 1-20.
 
\bibitem{KrylovITO} N. V. Krylov, 
On the It\^o-Wentzell formula for distribution-valued processes and related topics. 
Probab. Theory Related Fields 150 (2011), no. 1-2, 295-319. 

\bibitem{KrylovLP} N. V. Krylov, An analytic approach to SPDEs, in: Stochastic Partial Differential Equations : Six Perspectives, in: AMS Mathematical surveys an Monographs, vol. 64, pp. 185-242.


\bibitem{KSEE}  N. V. Krylov, B.L. Rozovskii, Stochastic evolution equations [MR0570795]. Stochastic differential equations: theory and applications, 1–69, Interdiscip. Math. Sci., 2, World Sci. Publ., Hackensack, NJ, 2007.

\bibitem{KRSAF} N.V. Krylov, M. V. Safonov,  A property of the solutions of parabolic equations with measurable coefficients. (Russian) Izv. Akad. Nauk SSSR Ser. Mat. 44 (1980), no. 1, 161-175, 239. 

\bibitem{KRU}	S. N. Kru\v{z}hkov, A priori bounds and some properties of solutions of elliptic and parabolic equations
Mat. Sb. (N.S.), 65(107):4 (1964),  522-570.

\bibitem{KUK} S. B. Kuksin, N. S. Nadirashvili, A. L.  Pyatnitski\v{i}, H\"older norm estimates for solutions of stochastic partial differential equations. (Russian) Teor. Veroyatnost. i Primenen. 47 (2002), no. 1, 152-159; translation in Theory Probab. Appl. 47 (2003), no. 1, 157-164.

\bibitem{LA} O. A. Lady\v{z}enskaja, V. A. Solonnikov, N. N. Ural'ceva,  Linear and quasilinear equations of parabolic type. (Russian) Translated from the Russian by S. Smith. Translations of Mathematical Monographs, Vol. 23 American Mathematical Society, Providence, R.I. 1968 xi+648 pp.

\bibitem{MOS} J. Moser,  On Harnack's theorem for elliptic differential equations. Comm. Pure Appl. Math. 14 (1961), 577-591.

\bibitem{MOS1} J. Moser, A Harnack inequality for parabolic differential equations. Comm. Pure Appl. Math. 17, 1964,  101-134.

\bibitem{NASH} J. Nash,  Continuity of solutions of parabolic and elliptic equations. Amer. J. Math. 80, 1958, 931-954.

\bibitem{BSPDE} J. Qiu and  S. Tang, Maximum principle for quasi-linear backward stochastic partial differential equations, Journal of Functional Analysis, Volume 262, Issue 5,  2012,  2436-2480.

\bibitem{Qiu2} J. Qiu, $L_2$-theory of linear degenerate SPDEs and $L_p$ $(p>0)$ estimates for the uniform norm of weak solutions, arXiv:1503.06162.

\bibitem{ROZ} B.L. Rozovskii, Stochastic evolution systems. Linear theory and applications to nonlinear filtering.  Mathematics and its Applications (Soviet Series), 35. Kluwer Academic Publishers Group, Dordrecht, (1990). xviii+315 pp.

\bibitem{SAF} M. V.  Safonov, Harnack's inequality for elliptic equations and H\"older property of their solutions. (Russian) Boundary value problems of mathematical physics and related questions in the theory of functions, 12. Zap. Nauchn. Sem. Leningrad. Otdel. Mat. Inst. Steklov. (LOMI) 96 (1980), 272-287, 312.

\bibitem{Wang} F. Wang, Harnack Inequalities for Stochastic Partial Differential Equations. SpringerBriefs in Mathematics (2013), 135 pp.






\end{thebibliography}


\end{document}